\documentclass[12pt,oneside]{amsproc}
\usepackage[T2A]{fontenc}
\usepackage[utf8]{inputenc}
\usepackage[english,russian]{babel}
\usepackage{amsmath,amssymb,amsthm}
\theoremstyle{plain}
\newtheorem{prop}{Утверждение}[section]
\newtheorem{theor}{Теорема}[section]

\numberwithin{equation}{section}
\numberwithin{subsection}{section}

\newcommand{\Wo}{{\raisebox{0.2ex}{\(\stackrel{\circ}{W}\)}}{}}
\newcommand{\ind}{\operatorname{ind}}
\newcommand{\Let}{\operatorname{Let}}

\title{Спектральные и осцилляционные свойства одного линейного пучка
дифференциальных операторов четвёртого порядка}
\author{Ж.~Бен Амара, А.\,А.~Владимиров, А.\,А.~Шкаликов}
\thanks{Работа выполнена при поддержке Российского фонда фундаментальных
исследований (грант \No~10-01-00423).}
\begin{document}
\renewcommand{\proofname}{{\upshape Д\,о\,к\,а\,з\,а\,т\,е\,л\,ь\,с\,т\,в\,о.}}
\begin{flushleft}
\normalsize УДК~517.984
\end{flushleft}
\begin{abstract}
Статья посвящается изучению спектральных и осцилляционных свойств линейного
операторного пучка \(A-\lambda B\), где коэффициент \(A\) отвечает дифференциальному
выражению \((py'')''\), а коэффициент \(B\) "--- дифференциальному выражению
\(-y''+cry\). В частности, устанавливается, что все отрицательные собственные
значения пучка являются простыми, а число нулей отвечающих им собственных функций
при некоторых дополнительных условиях связано с порядковым номером соответствующего
собственного значения.
\end{abstract}
\maketitle
\markboth{}{}

\section{Введение}\label{par:1}
\subsection{}\label{pt:1}
Рассмотрим спектральную задачу, отвечающую дифференциальному уравнению
\begin{gather}\label{eq:1:1}
	(py'')''-\lambda(-y''+cry)=0\\
	\intertext{и какому-либо из наборов граничных условий}\label{eq:1:2}
	y(0)=y'(0)=y(1)=y'(1)=0\\ \intertext{или} \label{eq:1:3}
	y(0)=y'(0)=y'(1)=(py'')'(1)+\lambda\alpha y(1)=0.
\end{gather}
Коэффициенты \(p,r\in C[0,1]\) предполагаются здесь равномерно положительными,
а физические параметры \(c\) и \(\alpha\) "--- вещественными. Решения \(y\)
предполагаются удовлетворяющими, помимо сформулированных граничных условий,
естественным ограничениям \(y\in C^2[0,1]\) и \(py''\in C^2[0,1]\).

Граничные задачи указанного вида возникают, в частности, в теории упругости,
описывая движение частично закреплённого стержня с сосредоточенной на свободном
конце дополнительной массой. Случай \(c\neq 0\) отвечает при этом наличию трения
стержня о текущую жидкость "--- например, "`течению"' стекла или пластика
по твёрдой подложке (см. \cite{He:1980}, \cite{GP:1997}). Другие примеры
механических приложений могут быть найдены в монографии \cite{Co:1948}.
В качестве гидродинамической интерпретации уравнения \eqref{eq:1:1} может быть
рассмотрено также хорошо известное одномерное уравнение Орра--Зоммерфельда без
мнимого промежуточного члена (см., например, \cite{12}, \cite{6}, \cite{8},
\cite{GP:1997}), возникающее в линеаризованной теории устойчивости течения
вязкой несжимаемой жидкости под действием силы тяжести.

\subsection{}
Граничные задачи, допускающие в операторной форме запись \(Ay-\lambda By=0\),
где \(A\) и \(B\) суть регулярные дифференциальные операторы порядков,
соответственно, \(n\) и \(m\) (при \(n>m\)), подвергались изучению, например,
в работах \cite{15}, \cite{16}, \cite{6} и \cite{9}. Однако в перечисленных
публикациях внимание было направлено на вопросы, связанные с полнотой систем
собственных и присоединённых функций пучка \(A-\lambda B\). В отличие от них,
основной целью нашего исследования является изучение вопросов, связанных
с простотой собственных значений рассматриваемой задачи, а также осцилляционными
свойствами её собственных функций. В этом смысле настоящая статья примыкает
к тематике работ \cite{GK:1950}--\cite{El:1978}. Следует сразу отметить,
что задача об изучении осцилляционных свойств собственных функций пучков
дифференциальных операторов четвёртого порядка является существенно более сложной,
чем аналогичная задача для операторов Штурма--Лиувилля.

Методологическую основу проводимого исследования составляют общие вариационные
принципы для линейных пучков самосопряжённых операторов (см., например,
\cite{LSY}) и теория знакорегулярных операторов в пространствах непрерывных
функций (см., например, \cite{GK:1950}, \cite{LS:1976}, \cite{BP:1994},
\cite{Vl:2009}). Основные результаты содержатся в пункте~\ref{par:3}.

\subsection{}\label{pt:1.2}
Введём в рассмотрение гильбертово пространство
\(\mathfrak H\) вида
\begin{equation}\label{eq:h1}
	\mathfrak H\rightleftharpoons\Wo_2^2[0,1]
\end{equation}
в случае граничных условий \eqref{eq:1:2}, и
\begin{equation}\label{eq:h2}
	\mathfrak H\rightleftharpoons\{y\in W_2^2[0,1]\::\:y(0)=y'(0)=y'(1)=0\}
\end{equation}
в случае граничных условий \eqref{eq:1:3}. Пространство \(\mathfrak H\)
естественным образом вложено в \(L_2[0,1]\), что позволяет ввести в рассмотрение
\cite[Дополнение~1, \S~2]{BSh:1983} оснащение
\[
	\mathfrak H\hookrightarrow L_2[0,1]\hookrightarrow\mathfrak H^*,
\]
где \(\mathfrak H^*\) "--- сопряжённое к \(\mathfrak H\) пространство полулинейных
функционалов. Обозначим через \(I:\mathfrak H\to L_2[0,1]\) и \(I^*:L_2[0,1]\to
\mathfrak H^*\) соответствующие операторы вложения, а через \(J:\mathfrak H^*\to
\mathfrak H\) "--- изометрию
\[
	(\forall y\in\mathfrak H^*)\,(\forall z\in\mathfrak H)\qquad
		\langle Jy,z\rangle_{\mathfrak H}=\langle y,z\rangle,
\]
существование которой гарантируется теоремой Рисса \cite[Гл.~\mbox{V}, \S~1]{Kato}
о представлении функционала в гильбертовом пространстве. Введём также в рассмотрение
линейный пучок \(T:\mathbb C\to\mathcal B(\mathfrak H,\mathfrak H^*)\), операторы
которого действуют согласно правилу
\begin{equation}\label{eq:qf}
	\langle T(\lambda)y,z\rangle\rightleftharpoons
		\int\limits_0^1\bigl[p\,y''\overline{z''}-\lambda\,(y'\overline{z'}
		+cr\,y\overline{z})\bigr]\,dx-\lambda\alpha\,y(1)\overline{z(1)}.
\end{equation}
Имеет место следующий простой факт (см., например, \cite{NSh1999} и \cite{Vl2004}):

\begin{prop}\label{prop:pochast}
При любом выборе значения \(\lambda\in\mathbb C\) ядро оператора \(T(\lambda)\)
в точности совпадает с множеством решений исходной граничной задачи
\eqref{eq:1:1}, \eqref{eq:1:2} или \eqref{eq:1:1}, \eqref{eq:1:3}.
\end{prop}

\begin{proof}
Установление факта принадлежности всякого решения \(y\) исходной задачи
пространству \(\mathfrak H\) и обращения при этом функционала \(T(\lambda)y\)
в тождественный нуль не вызывает затруднений. Пусть теперь \(y\in\mathfrak H\)
и \(T(\lambda)y=0\). Введём в рассмотрение функцию
\[
	w\rightleftharpoons py''+\lambda y-\lambda c\int\limits_0^x ry\cdot
		(x-t)\,dt.
\]
Выводимое из \eqref{eq:qf} интегрированием по частям тождество
\[
	(\forall z\in\Wo_2^2[0,1])\qquad\int\limits_0^1 w\overline{z''}\,dx=0
\]
означает, что функция \(w\) является линейной. Последнее, в свою очередь,
немедленно влечёт справедливость соотношений \(y\in C^2[0,1]\), \(py''\in
C^2[0,1]\) и равенства \eqref{eq:1:1}. Выполнение равенства
\[
	(py'')'(1)+\lambda\alpha y(1)=0
\]
в случае \eqref{eq:h2} проверяется интегрированием определения \eqref{eq:qf}
по частям для \(z\rightleftharpoons 2x^3-3x^2\).
\end{proof}

Утверждение~\ref{prop:pochast} позволяет рассматривать пучок \(T\) в качестве
адекватной операторной модели исходной граничной задачи, что и будет делаться
нами в дальнейшем.

\begin{prop}\label{prop:0}
Спектр пучка \(T\) является вещественным, дискретным и полупростым.
\end{prop}

\begin{proof}
Независимо от выбора значения \(\lambda\in\mathbb C\) ограниченная обратимость
оператора \(T(\lambda)\) с очевидностью равносильна ограниченной обратимости
оператора
\[
	JT(0)+\lambda\,J(dT/d\lambda):\mathfrak H\to\mathfrak H.
\]
При этом, как следует из представления \eqref{eq:qf}, оператор \(JT(0)\) является
равномерно положительным, а оператор \(J(dT/d\lambda)\) "--- самосопряжённым
и вполне непрерывным. Соответственно, значение \(\lambda=0\) принадлежит
резольвентному множеству пучка \(T\), а произвольное \(\lambda\neq 0\) попадает
в его спектр в том и только том случае, когда \(\lambda^{-1}\) принадлежит спектру
самосопряжённого вполне непрерывного оператора \(R\rightleftharpoons-[JT(0)]^{-1/2}
J(dT/d\lambda)[JT(0)]^{-1/2}\). Учёт того обстоятельства, что разрешимость уравнения
\[
	T(\lambda)y=(dT/d\lambda)z
\]
для произвольно фиксированного вектора \(z\in\ker T(\lambda)\setminus\{0\}\)
приводила бы к существованию присоединённых векторов у самосопряжённого оператора
\(R\), завершает доказательство.
\end{proof}

\subsection{}
Укажем на некоторые понятия и постановки задач, в том или ином смысле примыкающие
к задаче о спектральных свойствах пучка \(T\).

Во-первых, с пучком \(T\) могут быть связаны пучки замкнутых неограниченных
операторов \(T(\lambda)I^{-1}{I^*}^{-1}:\mathfrak H^*\to\mathfrak H^*\)
и \({I^*}^{-1}T(\lambda)I^{-1}:L_2[0,1]\to L_2[0,1]\). Из утверждения
\ref{prop:pochast} легко выводится, что спектр этих пучков (понимаемый
как множество значений параметра \(\lambda\in\mathbb C\), при которых значение
пучка не имеет ограниченного обратного), в точности совпадает со спектром
пучка \(T\). Тем самым, указанные переформулировки не привносят в задачу
дополнительного содержания.

Во-вторых, сделанный нами выбор трактовки задачи позволяет легко распространить
получаемые результаты на ту существенно более общую ситуацию, когда коэффициент
\(p\) представляет собой произвольную равномерно положительную функцию класса
\(L_{\infty}[0,1]\), а коэффициент \(r\) "--- произвольную знакоопределённую
обобщённую функцию класса \(W_2^{-1}[0,1]\). В настоящей статье, однако,
мы не станем заниматься этой проблематикой.

\subsection{}
Через \(\ind L\) на всём протяжении статьи обозначается отрицательный индекс
инерции квадратичной формы оператора \(L\), то есть точная верхняя грань
размерностей подпространств \(\mathfrak M\subseteq\mathfrak H\), удовлетворяющих
условию
\[
	(\exists\varepsilon>0)\,(\forall y\in\mathfrak M)\qquad
	\langle Ly,y\rangle\leqslant-\varepsilon\,\|y\|^2.
\]


\section{Модельная задача и допустимое множество}\label{par:2}
\subsection{}
Имеют место следующие два факта, фигурирующие в работе \cite{LN} как лемма~2.1
и лемма~2.2, соответственно:

\begin{prop}\label{lem:2:1}
Для любых равномерно положительных функций \(p,r\in C[0,1]\) любое
нетривиальное решение уравнения
\begin{equation}\label{eq:2:1}
	(py'')''-ry=0,
\end{equation}
удовлетворяющее при некотором \(a\in [0,1)\) условиям
\begin{align*}
	y(a)&\geqslant 0,&y'(a)&\geqslant 0,&y''(a)&\geqslant 0,&
	(py'')'(a)&\geqslant 0,\\ \intertext{удовлетворяет также неравенствам}
	y(1)&>0,& y'(1)&>0,& y''(1)&>0,& (py'')'(1)&>0.
\end{align*}
\end{prop}

\begin{prop}\label{lem:2:2}
Для любых равномерно положительных функций \(p,r\in C[0,1]\) любое
нетривиальное решение уравнения \eqref{eq:2:1}, удовлетворяющее при некотором
\(a\in (0,1]\) условиям
\begin{align*}
	y(a)&\geqslant 0,&y'(a)&\leqslant 0,&y''(a)&\geqslant 0,&
	(py'')'(a)&\leqslant 0,\\ \intertext{удовлетворяет также неравенствам}
	y(0)&>0,& y'(0)&<0,& y''(0)&>0,& (py'')'(0)&<0.
\end{align*}
\end{prop}

\subsection{}
Перед тем, как приступить непосредственно к исследованию свойств пучка \(T\),
рассмотрим вспомогательный линейный пучок \(S:\mathbb C\to\mathcal B(\mathfrak H,
\mathfrak H^*)\), отвечающий дифференциальному уравнению
\begin{equation}\label{eq:2:1a}
	(py'')''-\lambda ry=0
\end{equation}
и какому-либо из наборов граничных условий \eqref{eq:1:2} или \eqref{eq:1:3}.
Как и в случае с пучком \(T\), значения пучка \(S\) предполагаются заданными
посредством правила
\begin{equation}\label{eq:qfS}
	\langle S(\lambda)y,z\rangle\rightleftharpoons
		\int\limits_0^1\bigl[p\,y''\overline{z''}-\lambda ry\overline{z}
		\bigr]\,dx-\lambda\alpha\,y(1)\overline{z(1)}.
\end{equation}
Коэффициенты \(p,r\in C[0,1]\) по-прежнему считаются равномерно положительными.
Аналогично тому, как было сделано при доказательстве утверждения~\ref{prop:pochast},
может быть показано, что ядро оператора \(S(\lambda)\) совпадает с пространством
классических решений исходной граничной задачи для уравнения \eqref{eq:2:1a}.

\begin{prop}\label{prop:1:1}
Все положительные собственные значения пучка \(S\) и все расположенные
на интервале \((0,1)\) нули его собственных функций, отвечающих положительным
собственным значениям, являются простыми.
\end{prop}

\begin{proof}
Аналогично тому, как было сделано при доказательстве утверждения \ref{prop:0},
может быть установлен факт вещественности, дискретности и полупростоты спектра
пучка \(S\). Пусть теперь \(\lambda>0\) "--- кратное собственное значение
указанного пучка. В силу вышесказанного, пространство классических решений
соответствующей граничной задачи для уравнения \eqref{eq:2:1a} оказывается тогда
не менее, чем двумерным. Соответственно, внутри него должна найтись нетривиальная
функция, удовлетворяющая условиям
\[
	y(0)=y'(0)=y''(0)=y'(1)=0,
\]
что противоречит утверждению \ref{lem:2:1}.

Наконец, пусть \(a\in (0,1)\) "--- кратный нуль собственной функции пучка \(S\),
отвечающей собственному значению \(\lambda>0\). Без ограничения общности при этом
можно считать выполненным неравенство \(y''(a)\geqslant 0\). Но тогда оказывается
существующим нетривиальное решение уравнения \eqref{eq:2:1a}, удовлетворяющее
какому-либо из наборов условий
\begin{align*}
	&y(a)=y'(a)=y'(1)=0,& y''(a)&\geqslant 0,&(py'')'(a)&\geqslant 0,\\
	&y(a)=y'(a)=y'(0)=0,& y''(a)&\geqslant 0,&(py'')'(a)&\leqslant 0,
\end{align*}
что также противоречит утверждениям \ref{lem:2:1} и \ref{lem:2:2}.
\end{proof}

\begin{prop}\label{prop:1:2}
Пусть \(\alpha\geqslant 0\). Тогда оператор \(-[S(0)]^{-1}(dS/d\lambda):
\mathfrak H\to\mathfrak H\) не увеличивает числа перемен знака никакой
вещественнозначной функции \(f\in\mathfrak H\).
\end{prop}

\begin{proof}
Путём рассуждения, аналогичного проведённому при доказательстве утверждения
\ref{prop:pochast}, легко устанавливается, что решение уравнения \(S(0)y=
-(dS/d\lambda)f\) представляет собой классическое решение граничной задачи,
отвечающей уравнению
\[
	(py'')''=rf,
\]
а также граничным условиям \eqref{eq:1:2}, если исходный пучок \(S\) отвечает
тем же граничным условиям, и граничным условиям
\begin{equation}\label{eq:1:3x}
	y(0)=y'(0)=y'(1)=(py'')'(1)+\alpha f(1)=0,
\end{equation}
если исходный пучок \(S\) отвечает граничным условиям \eqref{eq:1:3}. Предположим,
что функция \(y\) имеет на интервале \((0,1)\) не менее \(n\) перемен знака,
то есть что найдутся \(n+1\)~точек
\[
	0<x_1<\ldots<x_{n+1}<1,
\]
удовлетворяющих при каждом \(k\in [1,n]\) неравенству \(y(x_k)y(x_{k+1})<0\).
Тогда будут справедливы следующие утверждения, последовательно устанавливаемые
на основе теоремы Лагранжа о среднем значении:
\begin{enumerate}
\item Функция \(y'\) имеет не менее \(n+1\) перемен знака в случае \eqref{eq:1:2},
и не менее \(n\) перемен знака в случае \eqref{eq:1:3x}.
\item Функции \(y''\) и \(py''\) имеют не менее \(n+2\) перемен знака в случае
\eqref{eq:1:2}, и не менее \(n+1\) перемен знака в случае \eqref{eq:1:3x}.
\item Функция \((py'')'\) имеет не менее \(n+1\) перемен знака в случае
\eqref{eq:1:2}, и не менее \(n\) перемен знака в случае \eqref{eq:1:3x}.
\end{enumerate}
В случае \eqref{eq:1:2} функция \(f=(py'')''/r\) потому заведомо имеет не менее
\(n\) перемен знака. В случае же \eqref{eq:1:3x} найдутся \(n+1\)~точек
\[
	0<\xi_1<\ldots<\xi_{n+1}<1,
\]
удовлетворяющих при каждом \(k\in [1,n]\) неравенству \((py'')'(\xi_k)\cdot
(py'')'(\xi_{k+1})<0\), а также \(n\)~точек \(\zeta_k\in (\xi_k,\xi_{k+1})\),
удовлетворяющих при каждом \(k\in [1,n]\) неравенству \((py'')'(\xi_k)\cdot
f(\zeta_k)<0\). При этом, согласно \eqref{eq:1:3x}, выполняется хотя бы одно
из неравенств \((py'')'(\xi_{n+1})\cdot (py'')'(1)\leqslant 0\) или
\((py'')'(\xi_{n+1})\cdot f(1)<0\), а потому найдётся точка \(\zeta_{n+1}\in
(\xi_{n+1},1)\) со свойством \((py'')'(\xi_{n+1})\cdot f(\zeta_{n+1})<0\).
Тем самым, в случае \eqref{eq:1:3x} функция \(f\) также имеет не менее \(n\)
перемен знака.
\end{proof}

Заметим, что в случае \(\alpha\geqslant 0\) при любом \(\lambda\leqslant 0\)
заданная определением \eqref{eq:qfS} квадратичная форма оператора \(S(\lambda)\)
является равномерно положительной на пространстве \(\mathfrak H\). Соответственно,
все собственные значения пучка \(S\) являются тогда положительными и подпадают
под действие утверждения \ref{prop:1:1}. Расположим спектр пучка \(S\)
в возрастающую последовательность
\[
	0<\lambda_1<\ldots<\lambda_n<\ldots
\]
простых положительных собственных значений. Зафиксируем также некоторую
последовательность \(\{y_n\}_{n=1}^{\infty}\) отвечающих собственным значениям
\(\lambda_n\) нормированных (в пространстве \(\mathfrak H\)) вещественнозначных
собственных функций. Поскольку оператор \([S(0)]^{-1}(dS/d\lambda)\) подобен
самосопряжённому оператору \([JS(0)]^{-1/2}J(dS/d\lambda)[JS(0)]^{-1/2}\),
функциональная последовательность \(\{y_n\}_{n=1}^{\infty}\) образует базис Рисса
в замыкании области значений оператора \([S(0)]^{-1}(dS/d\lambda)\).

\begin{prop}
Пусть \(\alpha\geqslant 0\). Тогда каждая собственная функция \(y_n\) имеет
в точности \(n-1\)~нулей на интервале \((0,1)\).
\end{prop}

\begin{proof}
Рассмотрим произвольную вещественнозначную функцию вида
\begin{equation}\label{eq:funk2}
	f=\sum\limits_{k=1}^{n}c_ky_k.
\end{equation}
Функциональная последовательность \(\{f_m\}_{m=1}^{\infty}\) вида
\[
	f_m\rightleftharpoons\sum\limits_{k=1}^n\left(\dfrac{\lambda_k}{%
	\lambda_n}\right)^mc_ky_k
\]
сходится в пространстве \(C^2[0,1]\) к функции \(c_ny_n\). При этом из утверждения
\ref{lem:2:1} вытекает справедливость неравенства \(y''_n(0)\neq 0\),
а из утверждения \ref{lem:2:2} вытекает справедливость неравенства \(y''_n(1)\neq
0\) в случае граничных условий \eqref{eq:1:2}, и неравенства \(y_n(1)\neq 0\)
в случае граничных условий \eqref{eq:1:3}. Из этих фактов и утверждения
\ref{prop:1:1} следует, что при \(m\gg 1\) функции \(f_m\) имеют в случае
\(c_n\neq 0\) в точности столько же перемен знака, сколько у собственной функции
\(y_n\). С учётом тождества
\[
	f=\lambda_n^m\bigl\{-[S(0)]^{-1}(dS/d\lambda)\bigr\}^m f_m
\]
и утверждения~\ref{prop:1:2} сказанное означает, что функция \(f\) вида
\eqref{eq:funk2} также имеет в указанном случае не большее число перемен знака.

Ввиду линейной независимости системы собственных функций пучка \(S\), при любом
\(n\geqslant 1\) найдётся набор коэффициентов \(\{c_k\}_{k=1}^n\), для которого
соответствующая функция \eqref{eq:funk2} будет иметь не менее \(n-1\) перемен
знака. При этом всегда можно добиться выполнения неравенства \(c_n\neq 0\).
Тогда, с учётом вышесказанного, функция \(y_n\) не может иметь менее \(n-1\)
перемен знака.

Рассмотрим теперь произвольную вещественнозначную функцию вида
\begin{equation}\label{eq:funk1}
	f=\sum\limits_{k=n}^{\infty}c_ky_k.
\end{equation}
Функциональная последовательность \(\{f_m\}_{m=1}^{\infty}\) вида
\[
	f_m\rightleftharpoons\lambda_n^m\bigl\{-[S(0)]^{-1}(dS/d\lambda)\bigr\}^mf
		=\sum\limits_{k=n}^{\infty}\left(\dfrac{\lambda_n}{\lambda_k}
		\right)^mc_ky_k
\]
сходится в пространстве \(\mathfrak H\) к функции \(c_ny_n\). Соответственно,
в случае \(c_n\neq 0\) при \(m\gg 1\) функции \(f_m\) не могут иметь меньшее число
перемен знака, чем у собственной функции \(y_n\). Согласно утверждению
\ref{prop:1:2} это означает, что функция \(f\) вида \eqref{eq:funk1} также имеет
в указанном случае не меньшее число перемен знака.

Заметим теперь, что вложенная в пространство \(\mathfrak H\) линейная оболочка
набора многочленов \(\{x^{k+2}\cdot (1-x)^2\}_{k=0}^{n-1}\) имеет размерность \(n\)
и не содержит функций с более чем \(n-1\) знакопеременами. Соответственно, среди
всевозможных результатов действия оператора \([S(0)]^{-1}(dS/d\lambda)\)
на указанные многочлены найдётся нетривиальная функция вида \eqref{eq:funk1}
с не превосходящим \(n-1\) числом перемен знака. С учётом вышесказанного
это означает существование номера \(N\geqslant n\), отвечающая которому собственная
функция \(y_N\) имеет не более \(n-1\) перемен знака.

Объединяя полученные результаты, устанавливаем, что при любом \(n\geqslant 1\)
собственная функция \(y_n\) имеет в точности \(n-1\) перемен знака на интервале
\((0,1)\). Ввиду гарантированной утверждением~\ref{prop:1:1} простоты нулей
этой функции, последнее равносильно доказываемому утверждению.
\end{proof}

\begin{prop}\label{prop:2:4}
Пусть \(\alpha\geqslant 0\). Тогда для любой собственной пары
\(\{\lambda,y\}\) пучка \(S\) функция \(y\) имеет в точности \(\ind S(\lambda)\)
нулей на интервале \((0,1)\).
\end{prop}

Данное утверждение представляет собой тривиальное следствие предыдущего
и известных вариационных принципов (см., например, \cite[Proposition~6]{LSY}).

\subsection{}\label{pt:2.3}
Вернёмся теперь к рассмотрению исходного операторного пучка \(T\). Назовём его
\emph{допустимым множеством} \(\Let(p)\) множество таких значений параметра
\(\lambda\in\mathbb R\), для которых квадратичная форма
\begin{equation}\label{eq:6}
	\int\limits_0^1 [p\,|y'|^2-\lambda |y|^2]\,dx
\end{equation}
равномерно положительна на \(\Wo_2^1[0,1]\). Это множество зависит только
от выбора коэффициента \(p\) и представляет собой бесконечную влево открытую
полупрямую, имеющую в качестве своей правой границы наименьшее собственное
значение задачи
\begin{gather}\label{eq:7}
	-(py')'-\lambda y=0,\\ \label{eq:8}
	y(0)=y(1)=0.
\end{gather}
Нашей ближайшей целью является установление того факта, что при всяком
\(\lambda\in\Let(p)\) квадратичная форма оператора \(T(\lambda)\) может быть
посредством замены переменной превращена в квадратичную форму некоторого
оператора рассмотренного выше модельного типа.

Как хорошо известно \cite[Теорема~1.6.2]{Po2009}, для каждого \(\lambda\in\Let(p)\)
можно зафиксировать равномерно положительное решение \(\sigma\in C^1[0,1]\)
дифференциального уравнения
\begin{equation}\label{eq:sturm}
	-(p\sigma')'-\lambda\sigma=0.
\end{equation}
Замена параметра
\[
	t(x)\rightleftharpoons\dfrac{1}{\omega}\int\limits_0^x\sigma\,d\xi,\qquad
	\omega\rightleftharpoons\int\limits_0^1\sigma\,d\xi
\]
определяет при этом непрерывную биекцию \(V:\mathfrak H\to\mathfrak H\),
сопоставляющую каждой функции \(y\in\mathfrak H\) функцию \(z\in\mathfrak H\) вида
\begin{align*}
	z(x)&\equiv y(t(x)),\\ z'(x)&\equiv y'(t(x))\dfrac{\sigma(x)}{\omega},\\
	z''(x)&\equiv y''(t(x))\dfrac{\sigma^2(x)}{\omega^2}+
	y'(t(x))\dfrac{\sigma'(x)}{\omega}.
\end{align*}

\begin{prop}\label{prop:2:1}
Оператор \(\hat T\rightleftharpoons V^*T(\lambda)V\) удовлетворяет тождеству
\[
	\langle\hat Ty,z\rangle\equiv\int\limits_0^1\left[\hat py''\overline{z''}
		-\hat ry\overline{z}\right]\,dx-\lambda\alpha\,y(1)\overline{z(1)},
\]
где функции \(\hat p\) и \(\hat r\) имеют вид
\begin{align}\label{eq:hatp}
	\hat p(t(x))&\equiv p(x)\dfrac{\sigma^3(x)}{\omega^3},\\
	\label{eq:hatr}
	\hat r(t(x))&\equiv\lambda cr(x)\dfrac{\omega}{\sigma(x)}.
\end{align}
\end{prop}

\begin{proof}
Повторим почти дословно рассуждения из доказательства теоремы~2.1 работы
\cite{BV:2006}. А именно, заметим, что для любой функции \(w\in\Wo_2^1[0,1]\)
выполняется равенство
\[
	\int\limits_0^1[p\sigma'\overline{w'}-\lambda\sigma w]\,dx=0.
\]
Полагая \(w\rightleftharpoons |z'|^2/\sigma\), \(z\rightleftharpoons Vy\),
получаем отсюда
\[
	\int\limits_0^1\left[p\sigma'\left(\dfrac{|z'|^2}{\sigma}\right)'-
	\lambda\,|z'|^2\right]\,dx=0.
\]
С учётом непосредственно проверяемого тождества
\[
	|z''(x)|^2-|y''(t(x))|^2\,\dfrac{\sigma^4(x)}{\omega^4}\equiv
	\sigma'(x)\,\left(\dfrac{|z'|^2}{\sigma}\right)'(x)
\]
устанавливаем теперь справедливость не зависящих от выбора функции \(y\in
\mathfrak H\) равенств
\begin{flalign*}
	&& \langle\hat Ty,y\rangle&=\langle T(\lambda)z,z\rangle\\
	&& &=\int\limits_0^1[p\,|z''|^2-\lambda(|z'|^2+cr\,|z|^2)]\,dx-
		\lambda\alpha\,|z(1)|^2\\
	&& &=\int\limits_0^1\hat p\,|y''|^2\,dt+\int\limits_0^1\left[p\sigma'
	\left(\dfrac{|z'|^2}{\sigma}\right)'-\lambda\,|z'|^2\right]\,dx-
	\int\limits_0^1\hat r\,|y|^2\,dt-\lambda\alpha\,|y(1)|^2\\
	&& &=\int\limits_0^1[\hat p\,|y''|^2-\hat r\,|y|^2]\,dt-
		\lambda\alpha\,|y(1)|^2.&&
\end{flalign*}
Ввиду принципа поляризации \cite[Гл.~I, \mbox{(6.11)}]{Kato}, это означает
справедливость доказываемого утверждения.
\end{proof}

\begin{prop}\label{prop:2:3}
Никакое число \(\lambda\in\Let(p)\) не может одновременно являться собственным
значением как пучка, отвечающего граничным условиям \eqref{eq:1:2}, так и пучка,
отвечающего граничным условиям \eqref{eq:1:3}.
\end{prop}

\begin{proof}
Предположим, что некоторая величина \(\lambda\in\Let(p)\) удовлетворяет требованиям
из формулировки доказываемого утверждения. Зафиксируем отвечающее ей равномерно
положительное решение \(\sigma\in C^1[0,1]\) уравнения \eqref{eq:sturm}, а также
связанные с ним функции \(\hat p\) и \(\hat r\) вида \eqref{eq:hatp},
\eqref{eq:hatr}. Утверждение~\ref{prop:2:1} означает, что функция \(z\in C^2[0,1]\)
принадлежит ядру оператора \(T(\lambda)\), отвечающего случаю граничных условий
\eqref{eq:1:2}, в том и только том случае, когда функция \(y\in C^2[0,1]\)
со свойством \(z(x)\equiv y(t(x))\) является классическим решением граничной
задачи, отвечающей уравнению
\begin{equation}\label{eq:2:100}
	(\hat py'')''-\hat ry=0
\end{equation}
и граничным условиям \eqref{eq:1:2}. Аналогичным образом, функция \(z\in C^2[0,1]\)
принадлежит ядру оператора \(T(\lambda)\), отвечающего случаю граничных условий
\eqref{eq:1:3}, в том и только том случае, когда функция \(y\in C^2[0,1]\)
со свойством \(z(x)\equiv y(t(x))\) является классическим решением граничной
задачи, отвечающей уравнению \eqref{eq:2:100} и граничным условиям 
\begin{equation}\label{eq:2:101}
	y(0)=y'(0)=y'(1)=(\hat py'')'(1)+\lambda\alpha y(1)=0.
\end{equation}

Предположение о нетривиальной разрешимости задачи \eqref{eq:2:100}, \eqref{eq:1:2}
означает равномерную положительность определённой соотношением \eqref{eq:hatr}
функции \(\hat r\). Соответственно, предположение о существовании нетривиальной
функции \(y\in C^2[0,1]\), удовлетворяющей равенству \eqref{eq:2:100} и каждому
из наборов условий \eqref{eq:1:2} и \eqref{eq:2:101}, противоречит утверждению
\ref{lem:2:2}. Иначе говоря, гипотеза из формулировки доказываемого утверждения
могла бы выполняться лишь тогда, когда множество решений граничной задачи,
отвечающей уравнению \eqref{eq:2:100} и граничным условиям
\[
	y(0)=y'(0)=y'(1)=0,
\]
было бы не менее чем двумерно. Однако в таком случае любое решение уравнения
\eqref{eq:2:100} с начальными условиями \(y(0)=y'(0)=0\) должно было бы
удовлетворять равенству \(y'(1)=0\), что противоречит утверждению \ref{lem:2:1}.
\end{proof}

\begin{prop}
Любое собственное значение \(\lambda\in\Let(p)\) пучка \(T\) имеет геометрическую
кратность \(1\).
\end{prop}

Это утверждение немедленно вытекает из утверждения~\ref{prop:2:3}.

\begin{prop}
Пусть \(\lambda\in\Let(p)\) "--- собственное значение пучка \(T\), удовлетворяющее
условиям \(\lambda c>0\) и \(\lambda\alpha\geqslant 0\). Тогда отвечающая ему
собственная функция имеет в точности \(\ind T(\lambda)\) простых нулей на интервале
\((0,1)\).
\end{prop}

Это утверждение немедленно вытекает из утверждений \ref{prop:2:1},
\ref{prop:2:4} и очевидного факта сохранения числа и простоты нулей функции
при действии оператора \(V\).


\section{Основные результаты}\label{par:3}
\subsection{}
Ввиду вытекающей из представления \eqref{eq:qf} равномерной положительности
квадратичной формы оператора \(T(0)\), все положительные собственные значения
пучка \(T\) имеют отрицательный тип, а все отрицательные "--- положительный.
Как следует из предложения \cite[Proposition~6]{LSY}, это означает справедливость
следующего утверждения:

\begin{prop}\label{prop:3:1}
Пусть \(\{\lambda_n\}_{n=1}^{\infty}\) "--- последовательность
сосчитанных в порядке возрастания с учётом кратности положительных собственных
значений пучка \(T\), а \(\{\lambda_{-n}\}_{n=1}^{\infty}\) "---
последовательность (возможно, частичная) сосчитанных в порядке убывания
с учётом кратности отрицательных собственных значений того же пучка. Тогда
при \(\lambda\in [0,\lambda_n]\) и \(\lambda\in [\lambda_{-n},0]\)
выполняются неравенства \(\ind T(\lambda)\leqslant n-1\), а при \(\lambda\in
(\lambda_n,+\infty)\) и \(\lambda\in (-\infty,\lambda_{-n})\) "--- неравенства
\(\ind T(\lambda)\geqslant n\).
\end{prop}

\subsection{}
Имеет место следующий факт:

\begin{theor}\label{prop:3:2}
При любом \(n\geqslant 1\) выполняется соотношение \(\lambda_{-n}\in\Let(p)\).
В случае граничных условий \eqref{eq:1:3} и выполнения неравенств \(c>0\)
и \(\alpha\geqslant 0\) выполняется также соотношение \(\lambda_1\in\Let(p)\).
\end{theor}

\begin{proof}
Принадлежность всех отрицательных собственных значений пучка \(T\) его допустимому
множеству вытекает из тривиального факта положительности наименьшего
собственного значения задачи \eqref{eq:7}, \eqref{eq:8}. Из утверждения
\ref{prop:3:1} вытекает также факт неотрицательности квадратичной формы
\eqref{eq:qf} при \(\lambda=\lambda_1\). Поскольку в случае граничных условий
\eqref{eq:1:3} множество всевозможных производных \(y'\) функций \(y\in\mathfrak H\)
в точности совпадает с пространством \(\Wo_2^1[0,1]\), сказанное означает
положительность квадратичной формы \eqref{eq:6} при \(c>0\), \(\alpha\geqslant 0\)
и \(\lambda=\lambda_1\).
\end{proof}

С учётом утверждения \ref{prop:3:1} и результатов предыдущего параграфа,
простыми следствиями этого факта являются следующие три теоремы:

\begin{theor}
Все отрицательные собственные значения пучка \(T\) являются простыми.
\end{theor}

\begin{theor}
Пусть \(c<0\) и \(\alpha\leqslant 0\). Тогда любая собственная функция
пучка \(T\), отвечающая его \(n\)-му отрицательному собственному значению
\(\lambda_{-n}\), имеет в точности \(n-1\) простых нулей на интервале \((0,1)\).
\end{theor}

\begin{theor}
В случае граничных условий \eqref{eq:1:3} и выполнения неравенств \(c>0\)
и \(\alpha\geqslant 0\) первое положительное собственное значение \(\lambda_1\)
пучка \(T\) является простым и имеет знакопостоянную на интервале \((0,1)\)
собственную функцию.
\end{theor}

\subsection{}
Имеет место следующий факт:

\begin{theor}
Число отрицательных собственных значений пучка \(T\) совпадает с числом
отрицательных собственных значений граничной задачи
\begin{gather*}
	-y''+cry=\lambda y,\\ y(0)=y(1)=0
\end{gather*}
в случае граничных условий \eqref{eq:1:2}, и граничной задачи
\begin{gather*}
	-y''+cry=\lambda y,\\ y(0)=y'(1)+\alpha y(1)=0
\end{gather*}
в случае граничных условий \eqref{eq:1:3}.
\end{theor}

\begin{proof}
Из утверждения \ref{prop:3:1} вытекает, что число отрицательных собственных
значений пучка \(T\) совпадает с величиной \(\lim\limits_{\lambda\to-\infty}
\ind T(\lambda)\). Последняя же, как следует из представления \eqref{eq:qf},
равна отрицательному индексу инерции заданной на пространстве \(\mathfrak H\)
квадратичной формы
\[
	\int\limits_0^1\bigl[|y'|^2+cr\,|y|^2\bigr]\,dx+\alpha\,|y(1)|^2.
\]
Отсюда и из того факта, что пополнением пространства \(\mathfrak H\) по норме
пространства \(W_2^1[0,1]\) является пространство \(\Wo_2^1[0,1]\) в случае
граничных условий \eqref{eq:1:2} и пространство
\[
	\{y\in W_2^1[0,1]\::\: y(0)=0\}
\]
в случае граничных условий \eqref{eq:1:3}, немедленно вытекает искомое.
\end{proof}

\subsection{}
На протяжении оставшейся части статьи через \(\lambda_n\) будут обозначаться
собственные значения пучка, отвечающего граничным условиям \eqref{eq:1:2}, а через
\(\lambda'_n\) "--- собственные значения пучка, отвечающего граничным условиям
\eqref{eq:1:3}.

\begin{theor}
При любом \(n\in\mathbb N\), для которого определены собственные
значения \(\lambda_{-n}\) и \(\lambda'_{-n}\), выполняется неравенство
\(\lambda_{-n}<\lambda'_{-n}\). При любом \(n\in\mathbb N\), для которого
определены собственные значения \(\lambda_{-n}\) и \(\lambda'_{-n-1}\),
выполняется неравенство \(\lambda'_{-n-1}<\lambda_{-n}\). В случае выполнения
неравенств \(c>0\) и \(\alpha\geqslant 0\) выполняется также неравенство
\(\lambda'_1<\lambda_1\).
\end{theor}

\begin{proof}
Согласно утверждению~\ref{prop:3:1}, для любой точки \(\lambda<\lambda_{-n}\)
квадратичная форма \eqref{eq:qf} отрицательна на некотором \(n\)-мерном
подпространстве пространства~\eqref{eq:h1}. Тогда она отрицательна и на некотором
\(n\)-мерном подпространстве более широкого пространства~\eqref{eq:h2},
что означает выполнение неравенства \(\lambda<\lambda'_{-n}\). Ввиду произвольности
выбора точки \(\lambda\), сказанное означает выполнение неравенства \(\lambda_{-n}
\leqslant\lambda'_{-n}\). Аналогичным образом устанавливается справедливость
неравенства \(\lambda'_1\leqslant\lambda_1\).

Далее, для любой точки \(\lambda<\lambda'_{-n-1}\) квадратичная форма
\eqref{eq:qf} отрицательна на некотором \((n+1)\)-мерном подпространстве
пространства~\eqref{eq:h2}. Поскольку пространство \eqref{eq:h1} имеет
в \eqref{eq:h2} коразмерность \(1\), то квадратичная форма \eqref{eq:qf}
отрицательна на некотором \(n\)-мерном подпространстве пространства~\eqref{eq:h2},
что означает выполнение неравенства \(\lambda<\lambda_{-n}\). Ввиду произвольности
выбора точки \(\lambda\), сказанное означает выполнение неравенства
\(\lambda'_{-n-1}\leqslant\lambda_{-n}\).

Для завершения доказательства остаётся обратиться к теореме \ref{prop:3:2}
и утверждению \ref{prop:2:3}.
\end{proof}


\end{document}